\documentclass[11pt]{article}%
\usepackage{mathrsfs}
\usepackage{amsfonts}
\usepackage{amsmath,amssymb}
\usepackage{graphicx}%
\usepackage{shorttoc}
\providecommand{\U}[1]{\protect \rule{.1in}{.1in}}

\newtheorem{theorem}{Theorem}[section]

\newtheorem{lemma}[theorem]{Lemma}
\newtheorem{proposition}[theorem]{Proposition}
\newtheorem{remark}[theorem]{Remark}

\newenvironment{proof}[1][Proof]{\noindent \textbf{#1.} }{\  $\Box$}
\numberwithin{equation}{section}

\newcommand{\E}{\widehat{\mathbb E}}
\renewcommand{\NG}{\mathcal N_G}
\newcommand{\Lip}{\operatorname{Lip}}
\newcommand{\erf}{\operatorname{erf}}
\newcommand{\R}{\mathbb R}
\newcommand{\norm}[1]{\left\|#1\right\|}

\begin{document}

\title{\textbf{A Berry--Esseen Bound of Order \((1+\log n)/\sqrt n\) for Peng's Central Limit Theorem
		 }}
\author{Yongsheng Song\thanks{State Key Laboratory of Mathematical Sciences, Academy of Mathematics and Systems Science, Chinese Academy of Sciences, Beijing 100190, China, and School of Mathematical Sciences, University of Chinese Academy of Sciences, Beijing 100049, China. Email: yssong@amss.ac.cn.}  }

\date{\today}
\maketitle

\begin{abstract}We derive a quantitative central limit theorem under sublinear expectations from a spatial endpoint regularity estimate for the one-dimensional
\(G\)-heat equation. For bounded Lipschitz initial data, the second spatial	derivative and the time derivative of its solution are Lipschitz in space and \(1/2\)-H\"older in time away from the initial time. In particular, \([u_{xx}(t,\cdot)]_{\mathrm{Lip}}\le C_G\Lip(\phi)/t\). Combining this estimate with the nonlinear Stein identity established by Song \cite{Song2020} and a replacement argument yields a bound of order \((1+\log n)/\sqrt n\) for bounded Lipschitz tests when the third absolute moment is finite and the lower variance is positive. The final replacement step is estimated directly, avoiding the nonintegrable singularity at time zero.

\end{abstract}

\textbf{Keywords}: Stein's method; central limit theorem; sublinear expectation; $G$-normal distribution; $G$-heat equation

\textbf{MSC classification}: 60F05; 60G50

\section{Introduction and main result}
Peng's central limit theorem \cite{Peng2008} identifies the \(G\)-normal
distribution as the limit of normalized sums of independent and
identically distributed random variables under a sublinear expectation.
Song \cite{Song2020} developed a nonlinear Stein method and obtained a
rate \(n^{-\gamma/2}\), where \(\gamma\in(0,1)\) is a spatial H\"older
exponent for the second derivative of the \(G\)-heat semigroup.
Krylov \cite[Theorem 1.1]{Krylov2020} obtained quantitative estimates
in a classical stochastic control formulation, allowing the lower
variance to be zero. For Lipschitz test functions and a finite third
absolute moment, his theorem gives a rate \(n^{-1/6}\).
Thus its scope includes the degenerate case, whereas the endpoint
regularity argument below requires a strictly positive lower variance.
The present argument uses the one-dimensional structure to reach
the spatial Lipschitz endpoint. Its time singularity is \(t^{-1}\),
which leads to a logarithmic loss in the replacement bound.

Let \((\Omega,\mathcal H,\E)\) be a sublinear expectation space:
\(\E\) is monotone, constant preserving, subadditive and positively
homogeneous. Independence is understood in Peng's directional sense:
if \(Y\) is independent of \(X\), then
\[
\E[f(X,Y)]=\E\bigl[\E[f(x,Y)]_{x=X}\bigr]
\]
for admissible test functions. This independence is directional: independence of \(Y\) from \(X\) need not imply independence of \(X\) from \(Y\). We assume that \(X_{i+1}\) is
independent of \((X_1,\ldots,X_i)\) and all \(X_i\) have the same law.
As in \cite{Song2020}, the one-variable law is represented on bounded
Lipschitz functions by a weakly compact family \(\Theta\) of probability
measures, with its moment extension:
\[
N[f]=\E[f(X_1)]=\sup_{\mu\in\Theta}\int f\,d\mu.\]
We assume that the relevant moment functions belong to the expectation
domain and that this representation is valid for them. For Borel
functions on \(\mathbb R\) with finite upper absolute expectation, we
also write \(N[f]=\sup_{\mu\in\Theta}\int f\,d\mu\). Thus the tail
expressions below are interpreted through the representing measures.
These representation and integrability hypotheses are in force
throughout the paper.

Suppose
\begin{equation}\label{eq:assumptions}
	\E[X_1]=\E[-X_1]=0,\qquad
	0<\alpha:=-\E[-X_1^2]\le\beta:=\E[X_1^2],\qquad
	m_3:=\E[|X_1|^3]<\infty.
\end{equation}
Define
\[
G(r)=\tfrac12(\beta r^+-\alpha r^-),\qquad
W_n=n^{-1/2}\sum_{i=1}^nX_i .
\]
For \(\phi\in C_{b,\mathrm{Lip}}(\R)\), let \(P_t^G\phi(x)=u(t,x)\),
where \(u\) is the unique bounded viscosity solution of
\begin{equation}\label{eq:pde}
	u_t-G(u_{xx})=0,\qquad u(0,x)=\phi(x).
\end{equation}
The \(G\)-normal expectation is \(\NG[\phi]=P_1^G\phi(0)\).
For a function \(f\), write
\[[f]_{\mathrm{Lip}}=\sup_{x\ne y}|f(x)-f(y)|/|x-y|.\]

\begin{theorem}[Main convergence bound]\label{thm:main}
	Under the preceding representation hypotheses and \eqref{eq:assumptions},
	for every bounded Lipschitz function
	\(\phi\), every integer \(n\ge1\), and
	\[
	A_G=\frac{4\beta}{\pi\alpha^2},
	\]
	one has
	\begin{equation}\label{eq:main}
		\left|\E[\phi(W_n)]-\NG[\phi]\right|
		\le \frac{\Lip(\phi)}{\sqrt n}
		\left(2A_Gm_3\log n+2\sqrt\beta\right).
	\end{equation}
	In particular, uniformly over bounded tests with Lipschitz constant at
	most one, the error is \(O((1+\log n)/\sqrt n)\).
\end{theorem}

The test class in this theorem measures a Lipschitz discrepancy between
sublinear laws; it is not a Kolmogorov distance for indicator tests.
The variance parameters \(\alpha,\beta\) are not H\"older exponents.
No continuous third spatial derivative of \(u\) is asserted.

\section{Endpoint regularity of the \(G\)-heat equation}

\subsection{A one-dimensional linear smoothing estimate}
\begin{lemma}\label{lem:smoothing}
	Let \(z\in C^{1,2}((s,+\infty)\times\mathbb{R})\cap C_b([s,+\infty)\times\mathbb{R})\) be a bounded classical solution of
	\[
	z_t=a(t,x)z_{xx},\qquad 0<\lambda\le a(t,x)\le\Lambda.
	\]
	
Then for \(t>s\),
	\begin{equation}\label{eq:lineargradient}
		\norm{z_x(t,\cdot)}_\infty
		\le \frac{\norm{z(s,\cdot)}_\infty}{\sqrt{\pi\lambda(t-s)}}.
	\end{equation}
	The constant does not involve derivatives of \(a\). The same
	conclusion holds on any finite time interval, assuming boundedness
	and continuity there.
\end{lemma}
\begin{proof}
Set \(K=\norm{z(s,\cdot)}_\infty\). If \(K=0\), the bounded-solution
maximum principle gives \(z=0\). Assume \(K>0\) and define, for $h>0$, $d\ge0$,
\[
 \Psi(h,d)=2K\erf\left(\frac{d}{4\sqrt{\lambda h}}\right),
 \qquad \erf(r)=\frac2{\sqrt\pi}\int_0^r e^{-q^2}\,dq.
\]
Then
\[
 \Psi_h=4\lambda\Psi_{dd},\quad \Psi_{dd}\le0,\quad
 \Psi(h,0)=0,\quad \lim_{h\downarrow0}\Psi(h,d)=2K\quad(d>0).
\]
We prove
\begin{equation}\label{eq:twopoint}
 z(s+h,y)-z(s+h,x)\le\Psi(h,y-x),\qquad x<y.
\end{equation}

Write
\begin {align}
F(h,x,y)
=\begin{cases}
	z(s+h,y)-z(s+h,x)-\Psi(h,y-x), &h>0, x\le y,\\
	z(s,y)-z(s,x)-2K,&h=0, x<y,\\
	0,&h=0, x=y.
\end{cases}
\end {align}
The extension \(F\) is upper semicontinuous on
\(D=[0,\infty)\times\{x\le y\}\), and continuous away from
\(\{(0,x,x):x\in\mathbb R\}\). At a corner, continuity of \(z\)
and \(\Psi\ge0\) imply \(\limsup F\le0\). 

Now, we check the boundary at \(h=0\) or \(x=y\). For \(h>0\), on the diagonal,
\[
F(h,x,x)=z(s+h,x)-z(s+h,x)-\Psi(h,0)
=0.
\]
At \(h=0\), for \(x<y\), $F(0,x,y)\le 0$ because
\[
z(s,y)-z(s,x)\le2K,
\] and \(F(0,x,x)\le 0\) by the definition.

On the whole space $D=\{h\ge0\}\times\{x\le y\}$, \(F\) need not attain its supremum. We must first add a penalty that forces a maximum to exist.

Let
\[
p(x)=\sqrt{1+x^2}.
\]
Then
\[
p(x)\longrightarrow\infty\quad(|x|\to\infty),
\qquad
0<p''(x)=\frac1{(1+x^2)^{3/2}}\le1.
\]
For \(\eta,\delta>0\), define
\begin {align}
F_{\eta,\delta}(h,x,y)
=F(h,x,y)-\eta\bigl(p(x)+p(y)\bigr)
-(2\Lambda\eta+\delta)h.
\end {align}
Fix $H>0.$ Because \(z\) is bounded on \([s,s+H]\times\mathbb R\) and \(\Psi\ge0\),
\[
F_{\eta,\delta}(h,x,y)\longrightarrow-\infty
\quad\text{as }|x|+|y|\longrightarrow\infty,
\]
uniformly in \(0<h\le H\). Thus any positive supremum must occur in a bounded spatial region.

Consequently, if \(F_{\eta,\delta}\) is positive anywhere, upper semicontinuity and spatial localization imply that it attains a positive maximum at some
\[
(h_0,x_0,y_0),\qquad 0<h_0\le H,\quad x_0<y_0.
\]
Write \(a_x=a(s+h_0,x_0)\), \(a_y=a(s+h_0,y_0)\), with all
following derivatives evaluated at this maximum.

We shall derive a contradiction at that maximum.

At this maximum, the spatial Hessian is negative semidefinite. Its entries are
\[
 (F_{\eta,\delta})_{xx}=-z_{xx}(s+h_0,x_0)-\Psi_{dd}-\eta p''(x_0),
\]
\[
 (F_{\eta,\delta})_{yy}=z_{xx}(s+h_0,y_0)-\Psi_{dd}-\eta p''(y_0),
 \qquad (F_{\eta,\delta})_{xy}=\Psi_{dd}.
\]
Taking the quadratic form with \(e=(-\sqrt{a_x},\sqrt{a_y})\) gives
\begin{align*}
 0&\ge e^\top D^2_{x,y}F_{\eta,\delta}e\\
 &=-a_xz_{xx}(s+h_0,x_0)+a_yz_{xx}(s+h_0,y_0)
 -(\sqrt{a_x}+\sqrt{a_y})^2\Psi_{dd}\\
 &\quad-\eta(a_xp''(x_0)+a_yp''(y_0)).
\end{align*}
Hence
\[-a_xz_{xx}(s+h_0,x_0)+a_yz_{xx}(s+h_0,y_0)\le
(\sqrt{a_x}+\sqrt{a_y})^2\Psi_{dd}
+\eta(a_xp''(x_0)+a_yp''(y_0)).\]
Using  \(z_t=a z_{xx}\), \(0<p''\le1\), and \(a_x,a_y\le\Lambda\),
we obtain
\begin{align}\label{Fh}
 (F_{\eta,\delta})_h
 &\le(\sqrt{a_x}+\sqrt{a_y})^2\Psi_{dd}-\Psi_h-\delta\notag\\
 &\le4\lambda\Psi_{dd}-\Psi_h-\delta=-\delta<0.
\end{align}
Here \((\sqrt{a_x}+\sqrt{a_y})^2\ge4\lambda\) and
\(\Psi_{dd}\le0\). But the time derivative at the maximum is
nonnegative: it is zero when \(h_0<H\), while at \(h_0=H\)
this follows from the left difference quotients. This contradicts
\eqref{Fh}. Hence \(F_{\eta,\delta}\le0\).

	For each fixed \(H\ge h>0\) and \(x<y\), we now have
\[
z(s+h,y)-z(s+h,x)
\le\Psi(h,y-x)+\eta(p(x)+p(y))
+(2\Lambda\eta+\delta)h.
\]
Letting \(\eta,\delta\downarrow0\) proves \eqref{eq:twopoint} on
\([0,H]\); \(H\) was arbitrary. Apply the same argument to \(-z\),
divide by \(y-x\), and let \(y\downarrow x\). The result is
\[
 |z_x(s+h,x)|\le\Psi_d(h,0)=\frac{K}{\sqrt{\pi\lambda h}}.
\]
\end{proof}

An alternative reflection coupling proof, together with a comparison
with \cite[Theorem 3.1]{Song2021}, is given in Appendix
\ref{app:coupling}. Its probabilistic construction assumes additional
local coefficient regularity, which holds in our smooth applications.

\begin{lemma}\label{lem:time}
	Under the hypotheses of Lemma \ref{lem:smoothing}, if \(z(s,\cdot)\)
	is \(K_1\)-Lipschitz, then
	\[
	|z(s+h,x)-z(s,x)|\le K_1\sqrt{2\Lambda h}.
	\]
\end{lemma}
\begin{proof}
Fix \(x_0\). Let \(W\) be a standard Brownian motion under the
linear expectation \(\mathbb E\), and define
\[
 B_\pm(h,x)=z(s,x_0)\pm K_1\mathbb E
 |x-x_0+\sqrt{2\Lambda}\,W_h|.
\]
For \(h>0\), both barriers solve \((B_\pm)_h=\Lambda(B_\pm)_{xx}\).
The upper barrier is convex and the lower barrier is concave, so
\[
 (B_+)_h-a(s+h,x)(B_+)_{xx}\ge0,\qquad
 (B_-)_h-a(s+h,x)(B_-)_{xx}\le0.
\]
Initially \(B_-(0,x)\le z(s,x)\le B_+(0,x)\).
Comparison on the whole space is justified by replacing \(B_+\) by
\(B_++\eta[(x-x_0)^2+2\Lambda h+1]\) and \(B_-\) by the
analogous negative perturbation. The added positive function is a
supersolution, since its parabolic operator is
\(2\eta(\Lambda-a)\ge0\), and the quadratic term localizes any
violation of comparison. Letting \(\eta\downarrow0\) gives
\(B_-\le z\le B_+\). At \(x=x_0\), therefore,
\[
 |z(s+h,x_0)-z(s,x_0)|
 \le K_1\sqrt{2\Lambda}\,\mathbb E|W_h|
 \le K_1\sqrt{2\Lambda h}.
\]
\end{proof}

\subsection{Uniform estimates for smooth approximations}
Let \(\rho\) be a nonnegative smooth mollifier supported in \([-1,1]\),
with integral one, and let \(\rho_\varepsilon(r)=
\varepsilon^{-1}\rho(r/\varepsilon)\). Set
\[
\phi_\varepsilon=\phi*\rho_\varepsilon,\qquad
G_\varepsilon(r)=(G*\rho_\varepsilon)(r)-(G*\rho_\varepsilon)(0).
\]
Then \(G_\varepsilon\) is smooth and convex, and
\begin{equation}\label{eq:approx}
	G_\varepsilon(0)=0,\quad
	\alpha/2\le G_\varepsilon'\le\beta/2,\quad
	\norm{G_\varepsilon-G}_\infty\le\beta\varepsilon,\quad
	\norm{\phi_\varepsilon-\phi}_\infty\le L\varepsilon,
\end{equation}
where \(L=\Lip(\phi)\).

Fix \(\varepsilon>0\). We construct a classical solution of
\begin{equation}\label{appro_equ}
 u^\varepsilon_t-G_\varepsilon(u^\varepsilon_{xx})=0,
 \qquad u^\varepsilon(0,x)=\phi_\varepsilon(x).
\end{equation}
Since \(G_\varepsilon'\ge\alpha/2>0\), the function
\(G_\varepsilon\) is a smooth bijection of \(\mathbb R\). Define
\[
 A_\varepsilon(r)
 =G_\varepsilon'\bigl(G_\varepsilon^{-1}(r)\bigr).
\]
Then \(A_\varepsilon\in C^\infty(\mathbb R)\) and
\(\alpha/2\le A_\varepsilon\le\beta/2\). Consider the scalar
uniformly parabolic quasilinear Cauchy problem
\begin{equation}\label{quasi_equ}
 w^\varepsilon_t=A_\varepsilon(w^\varepsilon)w^\varepsilon_{xx},
 \qquad
 w^\varepsilon(0,x)=G_\varepsilon(\phi_\varepsilon''(x)).
\end{equation}
The initial function belongs to \(C_b^\infty(\mathbb R)\), where
\(C_b^\infty\) denotes smooth functions whose derivatives of every
order, including order zero, are bounded.

Classical scalar quasilinear parabolic theory
\cite{LSU1968} gives a global bounded classical solution
\(w^\varepsilon\), smooth for positive times. In particular,
\(w^\varepsilon\) is continuous up to \(t=0\), and the smooth initial
data yield continuity of \(w^\varepsilon_t\),
\(w^\varepsilon_x\), and \(w^\varepsilon_{xx}\) up to \(t=0\)
on compact spatial sets and finite time intervals.
For completeness, the maximum principle bounds the range of
\(w^\varepsilon\) by the range of its initial value:
\[
 \|w^\varepsilon(t,\cdot)\|_\infty
 \le \|G_\varepsilon(\phi_\varepsilon'')\|_\infty.
\]
On this bounded range, all derivatives of \(A_\varepsilon\) are
bounded. Interior parabolic H\"older estimates, followed by linear
Schauder estimates and differentiation, provide the regularity needed
for continuation, excluding a finite maximal existence time.
These preliminary regularity bounds may depend on \(\varepsilon\).

Define
\[
 u^\varepsilon(t,x)
 =\phi_\varepsilon(x)+\int_0^t w^\varepsilon(r,x)\,dr.
\]
The chain rule and \eqref{quasi_equ} give
\[
 \partial_tG_\varepsilon^{-1}(w^\varepsilon)
 =\frac{w^\varepsilon_t}
 {G_\varepsilon'(G_\varepsilon^{-1}(w^\varepsilon))}
 =w^\varepsilon_{xx}.
\]
The initial-time regularity just stated permits differentiation under
the integral locally in \(x\). Thus
\begin{align*}
 u^\varepsilon_{xx}(t,x)
 &=\phi_\varepsilon''(x)+\int_0^t w^\varepsilon_{xx}(r,x)\,dr\\
 &=\phi_\varepsilon''(x)
   +G_\varepsilon^{-1}(w^\varepsilon(t,x))
   -G_\varepsilon^{-1}(w^\varepsilon(0,x))\\
 &=G_\varepsilon^{-1}(w^\varepsilon(t,x)).
\end{align*}
Consequently, \(u^\varepsilon_t=w^\varepsilon
=G_\varepsilon(u^\varepsilon_{xx})\), so \(u^\varepsilon\) solves
\eqref{appro_equ}. It is bounded on each finite time strip by its
construction. Comparison with constants and spatial translates then
implies
\[
 \|u^\varepsilon(t,\cdot)\|_\infty\le\|\phi_\varepsilon\|_\infty,
 \qquad
 \operatorname{Lip}(u^\varepsilon(t,\cdot))
 \le\operatorname{Lip}(\phi_\varepsilon)\le L.
\]
In particular, \(u^\varepsilon\) and \(u^\varepsilon_x\) are bounded
and continuous up to the initial time.

Set \(v^\varepsilon=u^\varepsilon_x\) and
\(a_\varepsilon=G_\varepsilon'(u^\varepsilon_{xx})
=A_\varepsilon(w^\varepsilon)\). The reconstruction identity yields
\[
 u^\varepsilon_{xxx}=\frac{w^\varepsilon_x}{a_\varepsilon},
 \qquad u^\varepsilon_{xt}=w^\varepsilon_x.
\]
Therefore
\begin{equation}\label{linear_equ}
 v^\varepsilon_t=a_\varepsilon v^\varepsilon_{xx},
 \qquad v^\varepsilon(0,x)=\phi_\varepsilon'(x),
\end{equation}
and \(w^\varepsilon_t=a_\varepsilon w^\varepsilon_{xx}\), with
\(\alpha/2\le a_\varepsilon\le\beta/2\).
Both \(v^\varepsilon\) and \(w^\varepsilon\) are bounded classical
solutions, continuous up to their initial time, so Lemma
\ref{lem:smoothing} applies to them. No preliminary global bounds on
their higher derivatives are needed for that lemma.
The estimates in Proposition \ref{prop:smooth} below are independent
of \(\varepsilon\).

\begin{proposition}\label{prop:smooth}
	For \(t>0\), the smooth solutions satisfy
	\begin{align}
		\norm{u^\varepsilon_x(t)}_\infty&\le L,\label{eq:ux}\\
		\norm{u^\varepsilon_{xx}(t)}_\infty
		&\le\frac{\sqrt2 L}{\sqrt{\pi\alpha t}},\label{eq:qbound}\\
		\norm{u^\varepsilon_t(t)}_\infty
		&\le\frac{\beta L}{\sqrt{2\pi\alpha t}},\label{eq:wbound}\\
		[u^\varepsilon_t(t)]_{\mathrm{Lip}}
		&\le\frac{2\beta L}{\pi\alpha t},\label{eq:wlip}\\
		[u^\varepsilon_{xx}(t)]_{\mathrm{Lip}}
		&\le\frac{4\beta L}{\pi\alpha^2t}.\label{eq:qlip}
	\end{align}
	For \(t,s\ge\tau>0\),
	\begin{align}
		|u^\varepsilon_t(t,x)-u^\varepsilon_t(s,x)|
		&\le\frac{2\beta^{3/2}L}{\pi\alpha\tau}|t-s|^{1/2},
		\label{eq:wtime}\\
		|u^\varepsilon_{xx}(t,x)-u^\varepsilon_{xx}(s,x)|
		&\le\frac{4\beta^{3/2}L}{\pi\alpha^2\tau}|t-s|^{1/2}.
		\label{eq:qtime}
	\end{align}
\end{proposition}
\begin{proof}
	Translation invariance and comparison preserve the initial Lipschitz
	constant, proving \eqref{eq:ux}. Put
	\[
	v=u^\varepsilon_x,\qquad w=u^\varepsilon_t,\qquad
	a=G_\varepsilon'(u^\varepsilon_{xx}).
	\]
	Differentiation gives
	\[
	v_t=a v_{xx},\qquad w_t=a w_{xx},\qquad
	\alpha/2\le a\le\beta/2.
	\]
	Apply Lemma \ref{lem:smoothing} to \(v\), starting at time zero,
	to obtain \eqref{eq:qbound}. Since
	\(|G_\varepsilon(r)|\le(\beta/2)|r|\), this gives
	\eqref{eq:wbound}. Apply that lemma to \(w\) between times \(t/2\)
	and \(t\):
	\[
	\norm{w_x(t)}_\infty
	\le\frac{\norm{w(t/2)}_\infty}{\sqrt{\pi\alpha t/4}}
	\le\frac{2\beta L}{\pi\alpha t}.
	\]
	The inverse of \(G_\varepsilon\) is \(2/\alpha\)-Lipschitz and
	\(u^\varepsilon_{xx}=G_\varepsilon^{-1}(w)\); hence \eqref{eq:qlip}.
	For \(t\ge s\ge\tau\), apply Lemma \ref{lem:time} to \(w\) starting
	at \(s\), using \eqref{eq:wlip} and \(\Lambda=\beta/2\).
	This proves \eqref{eq:wtime}; the inverse Lipschitz bound proves
	\eqref{eq:qtime}.
\end{proof}

\subsection{Passage to the viscosity solution}
\begin{theorem}[Endpoint regularity]\label{thm:regularity}
	For bounded Lipschitz initial data, the solution of \eqref{eq:pde}
	is classical on \((0,\infty)\times\R\). All the estimates in
	Proposition \ref{prop:smooth} hold with \(u^\varepsilon\) replaced by
	\(u\). In particular,
	\[
	[u_{xx}(t,\cdot)]_{\mathrm{Lip}}\le A_G L/t,
	\]where \[
	A_G=\frac{4\beta}{\pi\alpha^2}.  
	\]
	Both \(u_t\) and \(u_{xx}\) are locally Lipschitz in space and locally
	\(1/2\)-H\"older in time, uniformly on positive-time strips.
\end{theorem}
\begin{proof}
	Write \(\eta_\varepsilon=\norm{\phi_\varepsilon-\phi}_\infty\) and
	\(\delta_\varepsilon=\norm{G_\varepsilon-G}_\infty\).
	Viscosity test functions show that
	\[
	u(t,x)+\eta_\varepsilon+t\delta_\varepsilon,\qquad
	u(t,x)-\eta_\varepsilon-t\delta_\varepsilon
	\]
	are respectively a supersolution and a subsolution for the
	\(G_\varepsilon\)-equation. Comparison therefore yields
	\begin{equation}\label{eq:uniformconv}
		\sup_{[0,T]\times\R}|u^\varepsilon-u|
		\le\eta_\varepsilon+T\delta_\varepsilon
		\le(L+\beta T)\varepsilon.
	\end{equation}
	Proposition \ref{prop:smooth} gives local uniform boundedness and
	equicontinuity of \(w^\varepsilon=u^\varepsilon_t\) and
	\(q^\varepsilon=u^\varepsilon_{xx}\) for positive times.
	Arzel\`a--Ascoli gives locally uniformly convergent subsequences with
	continuous limits \(w,q\). For every compactly supported smooth test
	function \(\psi\) in the positive-time region, integration by parts
	and \eqref{eq:uniformconv} imply
	\[
	\int w\psi=-\int u\psi_t,\qquad
	\int q\psi=\int u\psi_{xx}.
	\]
	These continuous distributional derivatives are classical derivatives.
	For clarity, the time statement follows by passing to the limit in
	\(u^\varepsilon(t,x)-u^\varepsilon(s,x)=
	\int_s^t w^\varepsilon(r,x)\,dr\), that is,
	\[u(t,x)-u(s,x)=
	\int_s^t w(r,x)\,dr.\] Thus, $u_t=w.$ For space, fix \(a<b\) and use
	\[
	u^\varepsilon_x(t,a)=
	\frac{u^\varepsilon(t,b)-u^\varepsilon(t,a)
		-\int_a^b(b-r)q^\varepsilon(t,r)\,dr}{b-a}.
	\]
Denote this limit by \(D(t,a)\). For every \(x\in[a,b]\),
pass to the limit in the integral Taylor formula to obtain
\[
 u(t,x)=u(t,a)+(x-a)D(t,a)
       +\int_a^x(x-r)q(t,r)\,dr.
\]
Consequently,
\[
 u_x(t,x)=D(t,a)+\int_a^x q(t,r)\,dr,\qquad u_{xx}(t,x)=q(t,x).
\]
The formula defining \(D(t,a)\) shows that it is continuous in \(t\);
hence \(u_x\) is jointly continuous. Since the interval was arbitrary,
these identities hold throughout the positive-time region. Finally,
\(u_t=w\), \(u_{xx}=q\), and \(w=G(q)\), by the uniform convergence
of \(G_\varepsilon\).
Every subsequential derivative limit is uniquely determined by \(u\),
	so the whole families converge locally uniformly. All the estimates
	pass to the limit.
\end{proof}

\begin{remark}[The endpoint does not imply \(C^3\)]
	The following example is a translated and rescaled version of the
	periodic eigenfunction constructed by Song
	\cite[Section 2, equation (2.5) and Proposition 2.1]{Song2015}.
	That construction already gives a bounded function with a uniformly
	Lipschitz second derivative and a separated-variable solution of the
	\(G\)-heat equation. Here we use it to illustrate the failure of
	\(C^3\) regularity across the joining points.
	Assume \(\alpha<\beta\), put \(a_0=\pi\sqrt\alpha\) and
	\(p_0=\pi(\sqrt\alpha+\sqrt\beta)\), and extend periodically the function
	\[
	H(x)=
	\begin{cases}
		\sqrt\alpha\sin(x/\sqrt\alpha),&0\le x\le a_0,\\
		-\sqrt\beta\sin((x-a_0)/\sqrt\beta),&a_0\le x\le p_0.
	\end{cases}
	\]
	The values of \(H,H',H''\) match at the joins, \(H''\) is Lipschitz,
	and \(G(H'')=-H/2\). Hence \(e^{-t/2}H(x)\) solves
	\eqref{eq:pde}. However,
	\(H'''(a_0-)=1/\alpha\ne1/\beta=H'''(a_0+)\).
	The endpoint assertion is Lipschitz continuity of \(u_{xx}\), not
	existence of a continuous third derivative.
\end{remark}

\section{Nonlinear Stein estimates}
We recall the precise ingredients of \cite[Lemmas 4.1 and 4.3]{Song2020}.
Here and below, \(C_b^2(\mathbb R)\) denotes the functions whose
first two derivatives and the functions themselves are bounded and
continuous. For a bounded continuous test \(F\), set
\(\Theta_F=\{\mu\in\Theta:\mu[F]=N[F]\}\).
Weak compactness of \(\Theta\) ensures that \(\Theta_F\) is nonempty.
The finite third moment implies the uniform tail condition required in
that paper, since \(N[|x|\mathbf1_{\{|x|>R\}}]\le m_3/R^2\).

The lemma below is the endpoint (H\"older exponent equal to one) case of \cite[Lemma 4.3]{Song2020}.

\begin{lemma}[Stein defect]\label{lem:defect}
	If \(F\) is in \(C^2_b(\mathbb{R})\), with Lipschitz
	second derivative, then, for \(\mu\in\Theta_F\),
	\[
	\left|\mu\left[\tfrac{x}{2}F'(x)-G(F''(x))\right]\right|
	\le2[F'']_{\mathrm{Lip}}m_3.
	\]
\end{lemma}

\begin{lemma}[Interpolating identity]\label{lem:identity}
	For  \(F\in C_b^2(\mathbb{R})\) with Lipschitz
	second derivative,
	define
	\[
	F_s(y)=P_s^GF(\sqrt{1-s}\,y),\qquad
	\mathcal L_GF(y)=G(F''(y))-\tfrac y2F'(y).
	\]
	The nonlinear Stein identity of \cite[Lemma 4.1]{Song2020} is
	\[
	\NG[F]-N[F]=
	\int_0^1\frac1{1-s}
	\sup_{\mu\in\Theta_{F_s}}\mu[\mathcal L_GF_s]\,ds.
	\]
	Consequently, whenever the right-hand bound is integrable,
	\begin{equation}\label{eq:steinbound}
		|\NG[F]-N[F]|
		\le 2m_3\int_0^1\frac{[F_s'']_{\mathrm{Lip}}}{1-s}\,ds.
	\end{equation}
\end{lemma}
\begin{proof}
	The identity is the result cited from \cite{Song2020}; it also holds
	with the supremum replaced by the infimum. Lemma \ref{lem:defect}
	applied to its maximizing measures gives \eqref{eq:steinbound}.
	In the replacement argument with \(i<n\), the initial test is
already smoothed by the positive time \(1-i/n\). In fact,
\([F_s'']_{\mathrm{Lip}}/(1-s)\) is bounded by a constant times
\((1-s)^{1/2}\), because the remaining heat time is bounded below
by \(1-i/n>0\). Thus the integral used there is finite.
\end{proof}

\section{Proof of the main convergence theorem}
\begin{proof}[Proof of Theorem \ref{thm:main}]
	Let \(L=\Lip(\phi)\), \(S_0=0\), and
	\[
	S_i=n^{-1/2}\sum_{k=1}^iX_k,\qquad
	a_i=\E[u(1-i/n,S_i)]\quad(0\le i\le n).
	\]
	The endpoints are \(a_0=\NG[\phi]\) and \(a_n=\E[\phi(W_n)]\).
	For \(1\le i\le n\), define
	\[
	b_i(x)=N[u(1-i/n,x+\cdot/\sqrt n)],\quad
	c_i(x)=\NG[u(1-i/n,x+\cdot/\sqrt n)].
	\]
	Independence gives \(a_i=\E[b_i(S_{i-1})]\), while the semigroup
	property gives \(a_{i-1}=\E[c_i(S_{i-1})]\). Sublinearity implies
	\begin{equation}\label{eq:replace}
		|a_i-a_{i-1}|\le\sup_x|b_i(x)-c_i(x)|.
	\end{equation}
	For \(i<n\), fix \(x\) and use
	\(F(y)=u(1-i/n,x+y/\sqrt n)\).
	Positive homogeneity of \(G\), translation invariance and the semigroup
	property identify its Stein interpolation as
	\[
	F_s(y)=u\left(1-\frac in+\frac sn,\,
	x+\sqrt{\frac{1-s}{n}}\,y\right).
	\]
	Thus
	\[
	[F_s'']_{\mathrm{Lip}}=
	\left(\frac{1-s}{n}\right)^{3/2}
	[u_{xx}(1-i/n+s/n,\cdot)]_{\mathrm{Lip}} .
	\]
	Lemma \ref{lem:identity} now gives
	\begin{align*}
		|b_i(x)-c_i(x)|
		&\le\frac{2m_3}{n^{3/2}}\int_0^1
		(1-s)^{1/2}[u_{xx}(1-i/n+s/n)]_{\mathrm{Lip}}\,ds\\
		&\le\frac{2m_3}{\sqrt n}
		\int_{1-i/n}^{1-(i-1)/n}[u_{xx}(t)]_{\mathrm{Lip}}\,dt .
	\end{align*}
	For \(n\ge2\), summing these bounds, using \eqref{eq:replace} and
	Theorem \ref{thm:regularity}, yields
	\[
	\sum_{i=1}^{n-1}|a_i-a_{i-1}|
	\le\frac{2m_3}{\sqrt n}\int_{1/n}^1\frac{A_GL}{t}\,dt
	=\frac{2A_GLm_3\log n}{\sqrt n}.
	\]
	For the last step, let \(\xi(y)=y\) denote the coordinate function. Lipschitz continuity gives
	\[
	|b_n(x)-\phi(x)|\le\frac L{\sqrt n}N[|\xi|]
	\le\frac{L\sqrt\beta}{\sqrt n}.
	\]
	Likewise
	\[
	|c_n(x)-\phi(x)|\le
	\frac L{\sqrt n}\NG[|\xi|]\le\frac{L\sqrt\beta}{\sqrt n},
	\]
	because both \(N[\xi^2]\) and \(\NG[\xi^2]\) equal \(\beta\).
	Consequently \(|a_n-a_{n-1}|\le2L\sqrt\beta/\sqrt n\).
	The triangle inequality proves \eqref{eq:main}.
	For \(n=1\) only this last-step estimate is needed.
\end{proof}

\section{Scope of the bound}
The endpoint estimate improves the regularity input from an unspecified
exponent \(\gamma\in(0,1)\) to spatial Lipschitz continuity.
Nevertheless, substitution into an integral extending to time zero is
not justified: the bound \(A_GL/t\) is nonintegrable there.
Separating the last replacement step produces the cutoff \(t=1/n\)
and hence \(\log n\). This is an upper bound; the argument does not
establish that the logarithmic factor is optimal.

\clearpage
\appendix
\section{A reflection coupling proof of Lemma \ref{lem:smoothing}}
\label{app:coupling}
We give an alternative probabilistic proof for coefficients for which
reflection coupling is classically well posed. Specifically, on each
finite time strip assume additionally that \(a\) is continuous in time
and locally Lipschitz in space, uniformly on compact sets. These
conditions hold for the smooth coefficients used in Section 2.2.
No Lipschitz constant enters the estimate. The analytic proof in
Section 2.1 retains the full generality of Lemma \ref{lem:smoothing};
the present construction does not assert strong SDE well-posedness
for arbitrary measurable coefficients.

\subsection{Time reversal and reflection coupling}
Fix \(s\ge0\), \(h>0\), and \(x<y\). Write
\[
 \sigma_h(r,\xi)=\sqrt{2a(s+(h-r)_+,\xi)},\qquad r\ge 0.
\]
  For \(0\le r\le h\), this reverses the time argument in the
forward Cauchy problem; for \(r>h\), it extends the coefficient by
its value at time \(s\). We have
\(\sqrt{2\lambda}\le\sigma_h\le\sqrt{2\Lambda}\).
On a Wiener probability space construct, up to their first meeting,
\[
 dX_r=\sigma_h(r,X_r)\,dB_r,\quad X_0=x,
 \qquad dY_r=-\sigma_h(r,Y_r)\,dB_r,\quad Y_0=y.
\]
Let \(\tau=\inf\{r\ge0:X_r=Y_r\}\), and after \(\tau\) set
\(Y=X\). Local Lipschitz continuity and boundedness give existence,
pathwise uniqueness before meeting, and nonexplosion. The process
\[
 \widetilde B_r=\int_0^r
 (\mathbf1_{\{q\le\tau\}}(-1)+\mathbf1_{\{q>\tau\}})\,dB_q
\]
is Brownian by L\'evy's characterization, and
\(dY_r=\sigma_h(r,Y_r)d\widetilde B_r\). Thus both marginals have
the required diffusion laws.
To justify the representation, apply It\^o's formula to
\(z(s+h-r,X_r)\), stopped at \(r=h-\delta\) and upon exit from
a bounded spatial interval. The drift is
\(-z_t+\tfrac12\sigma_h(r,X_r)^2z_{xx}=0\).
Nonexplosion and boundedness of \(z\) allow removal of the spatial
stopping. Letting \(\delta\downarrow0\), continuity at time \(s\)
and dominated convergence give the representation. The same argument
applies to \(Y\). Consequently,
\begin{equation}\label{eq:coupling-representation}
 z(s+h,x)=\mathbb E[z(s,X_h)],\qquad
 z(s+h,y)=\mathbb E[z(s,Y_h)].
\end{equation}
Since \(X_h=Y_h\) on \(\{\tau\le h\}\),
\begin{equation}\label{eq:coupling-estimate}
\begin{aligned}
 |z(s+h,y)-z(s+h,x)|
 &=\left|\mathbb E\bigl[(z(s,Y_h)-z(s,X_h))
                  \mathbf1_{\{\tau>h\}}\bigr]\right|\\
 &\le2\|z(s,\cdot)\|_\infty\,\mathbb P(\tau>h).
\end{aligned}
\end{equation}
The next subsection estimates \(\mathbb P(\tau>h)\).

\subsection{The meeting-time estimate}
Set \(d=y-x\). Before meeting, \(Y_r-X_r=d+M_r\), where the stopped
martingale is
\[
 M_r=-\int_0^{r\wedge\tau}
       (\sigma_h(q,X_q)+\sigma_h(q,Y_q))\,dB_q.
\]
To apply the time-change theorem, extend this martingale after \(\tau\) with integrand
\(-\sqrt{8\lambda}\). Denote the extended martingale by
\(\overline M\) and its quadratic variation by \(Q\). Then
\[
 Q_r\ge8\lambda r,
\]
because before meeting the sum of the two diffusion coefficients is
at least \(2\sqrt{2\lambda}\). By the
Dambis--Dubins--Schwarz theorem,
\(\overline M_r=\beta_{Q_r}\) for a Brownian motion \(\beta\).
On \(\{\tau>h\}\), the distance has remained positive up to \(h\).
Since \(Q\) is continuous and strictly increasing, this implies
\[
 \min_{0\le q\le8\lambda h}(d+\beta_q)\ge \min_{0\le r\le h}(d+\overline M_r)>0.
\]
The reflection principle therefore gives, with \(\Phi\) the standard
normal distribution function,
\begin{equation}\label{eq:coupling-tail}
 \mathbb P(\tau>h)\le \mathbb{P}(\min_{0\le q\le8\lambda h}\beta_q>-d)
 =2\Phi\left(\frac{d}{\sqrt{8\lambda h}}\right)-1
 \le\frac{d}{2\sqrt{\pi\lambda h}}.
\end{equation}
Writing \(K=\|z(s,\cdot)\|_\infty\), 
\eqref{eq:coupling-estimate} yields
\[
 |z(s+h,y)-z(s+h,x)|
 \le2K\mathbb P(\tau>h)
 \le\frac{K|y-x|}{\sqrt{\pi\lambda h}}.
\]
Hence
\begin{equation}\label{eq:coupling-gradient}
 \|z_x(s+h,\cdot)\|_\infty
 \le\frac{\|z(s,\cdot)\|_\infty}{\sqrt{\pi\lambda h}},
 \qquad h>0.
\end{equation}
This is exactly the constant in Lemma \ref{lem:smoothing} and holds
for all \(h>0\). 
\subsection{Relation to Theorem 3.1 in \cite{Song2021}}
Song \cite[Theorem 3.1]{Song2021} treats multidimensional semilinear
semigroups with spatially Lipschitz diffusion and drift coefficients
and a Lipschitz BSDE generator \(g\). Its gradient estimate contains
\[
 \frac{e^{\mu T}}{\sqrt{(1-e^{-LT})/L}},\quad
 \mu=K_g+4L_g^2,\quad
 L=2(L_gL_\sigma+L_b+2L_\sigma^2).
\]
For \(g=0\), $b=0$, Corollary 3.3 has \(L=4L_\sigma^2\), which depends on the Lipschitz constant \(L_\sigma\) of the diffusion coefficient \(\sigma\). For \(0<T\le1\), converting
this estimate into a multiple of \(T^{-1/2}\) still introduces a
constant depending on \(L_{\sigma}\).

In the coupling used in the proof of that theorem, the difference
of the diffusion terms is \(\sigma(X)-\sigma(Y)\). Applying It\^o's
formula to squared distance introduces
\(\|\sigma(X)-\sigma(Y)\|^2\), estimated by
\(L_\sigma^2|X-Y|^2\); the BSDE change of measure also introduces
\(L_gL_\sigma\). In contrast, our one-dimensional, drift-free linear
problem permits reflection: the distance martingale has diffusion
coefficient \(\sigma_h(r,X)+\sigma_h(r,Y)\). Its positive lower
bound controls the meeting time directly, without differentiating
\(\sigma_h\).

Accordingly, this appendix supplies a stronger derivative-independent
estimate in the one-dimensional linear setting relevant here, and
also allows time-dependent coefficients. It does not establish the
same improvement for the full multidimensional semilinear theorem.
The nondecaying large-time upper bound of Corollary 3.3 when \(L>0\)
does not contradict \eqref{eq:coupling-gradient}: it is a different,
more general estimate and does not assert a matching lower bound.

\section*{Acknowledgements}

 The author is supported by the National Key R\&D Program of China (No. 2024YFA1013503) and the National Natural Science Foundation of China (No. 12431017).

%%% ----------------------------------------------------------------------

\renewcommand{\refname}{\large References}{\normalsize \ }


\begin{thebibliography}{99}

\bibitem{CIL1992}
Crandall M G, Ishii H, Lions P L.
User's guide to viscosity solutions of second order partial differential
equations. \emph{Bulletin of the American Mathematical Society},
1992, 27(1): 1--67.

\bibitem{Krylov2020}
Krylov N V.
On Shige Peng's central limit theorem.
\emph{Stochastic Processes and their Applications},
2020, 130(3): 1426--1434.


\bibitem{LSU1968}
Ladyzhenskaya O A, Solonnikov V A, Ural'tseva N N.
\emph{Linear and Quasilinear Equations of Parabolic Type}.
Translations of Mathematical Monographs, Vol. 23.
American Mathematical Society, Providence, RI, 1968.

\bibitem{Peng2008}
Peng S.
A new central limit theorem under sublinear expectations.
2008, arXiv:0803.2656.

\bibitem{Song2015}
Song Y.
A note on \(G\)-normal distributions.
\emph{Statistics and Probability Letters},
2015, 106: 142--146.


\bibitem{Song2020}
Song Y.
Normal approximation by Stein's method under sublinear expectations.
\emph{Stochastic Processes and their Applications},
2020, 130(5): 2838--2850.


\bibitem{Song2021}
Song Y S.
Gradient estimates for nonlinear diffusion semigroups by coupling methods.
\emph{Science China Mathematics}, 2021, 64(5): 1093--1108.
DOI: 10.1007/s11425-018-9541-6.

\bibitem{Wang1992}
Wang L.
On the regularity theory of fully nonlinear parabolic equations: II.
\emph{Communications on Pure and Applied Mathematics},
1992, 45(2): 141--178.


\end{thebibliography}
\end{document}